\documentclass[10pt, a4paper,reqno]{amsart}
\usepackage{amsthm,euscript,colonequals}
\usepackage{mathtools} %needed for \xmapsto
\usepackage{stmaryrd}%need for \mapsfrom
\theoremstyle{}
{\theoremstyle{definition}
\newtheorem{dfn}{Definition}[section]}
\newtheorem{prop}[dfn]{Proposition}

\newtheorem{thm}[dfn]{Theorem}
{\theoremstyle{definition}
}

\newtheorem{cor}[dfn]{Corollary}

{\theoremstyle{definition}
}

\usepackage[top=30truemm,bottom=30truemm,left=35truemm,right=35truemm]{geometry}
\usepackage{amsmath,amssymb,amscd,bm,graphicx,tikz-cd,upgreek,dsfont,amsfonts,tikz}
\usepackage[all]{xy}
\usepackage[colorlinks]{hyperref}

\tikzset{
        DB/.style={circle,draw=black,circle,fill=white,inner sep=0pt, minimum size=5pt},
        DW/.style={circle,draw=black,fill=black,inner sep=0pt, minimum size=5pt},
         gap/.style={inner sep=0.5pt,fill=white},
}

\usepackage{setspace}
\newcommand{\marginparstretch}{0.6}
\let\oldmarginpar\marginpar
\renewcommand\marginpar[1]{\-\oldmarginpar[\framebox{\setstretch{\marginparstretch}\begin{minipage}{\marginparwidth}{\raggedleft\tiny #1}\end{minipage}}]{\framebox{\setstretch{\marginparstretch}\begin{minipage}{\marginparwidth}{\raggedright\tiny #1}\end{minipage}}}}

\usepackage{enumitem}
\setlist[enumerate]{format=\normalfont}

\usepackage{array,multirow,booktabs,longtable}
\numberwithin{equation}{section}

\newcommand{\ii}{\kern 1pt {\rm i}\kern 1pt }

\newcommand{\Hom}{\operatorname{Hom}}

\newcommand{\End}{\operatorname{End}}

\newcommand{\Aut}{\operatorname{Aut}}

\renewcommand{\mod}{\operatorname{mod}}

\newcommand{\add}{\operatorname{add}}

\newcommand{\proj}{\operatorname{proj}}

\newcommand{\Kb}{\operatorname{K^b}}

\newcommand\Db{\mathop{\rm{D}^b}}

\newcommand{\tilt}{\operatorname{\sf tilt}}
\newcommand{\silt}{\operatorname{\sf silt}}

\newcommand{\cStab}[1]{\mathrm{Stab}_{#1}^{\kern -0.5pt \circ}\kern -0.25pt}
\newcommand{\TitsR}{{\sf Tits}_{\kern 1pt \bR}}
\newcommand{\CTitsR}{{\sf CTits}_{\kern 1pt \bR}}
\newcommand{\CTitsC}{{\sf CTits}_{\kern 1pt \bC}}
\newcommand{\cAut}{\mathrm{Aut}^{\kern 0pt \circ}\kern -0.25pt}

\newcommand{\cJ}{\mathcal{J}}

\newcommand{\bC}{\mathbb{C}}

\newcommand{\bL}{\mathbb{L}}

\newcommand{\bR}{\mathbb{R}}

\newcommand{\Delt}{\Updelta_{\kern 0.05em 0}}
\newcommand{\WDelt}{W_{\kern -0.1em \Updelta}}
\newcommand{\WDeltaff}{W_{\kern -0.1em \Updelta_{\aff}}}
\newcommand{\Wkern}[1]{W_{\kern -0.1em #1}\kern 0.05em}
\newcommand{\wo}[1]{w_{\kern -0.075em #1}}
\newcommand{\wop}[1]{w^{\phantom J}_{\kern -0.1em #1}}
\newcommand{\GammaJ}{\Upgamma_{\kern -0.05em J}}
\newcommand{\GammaS}{\Upgamma_{\kern -0.05em \cJ}}

\newcommand{\xlup}[1]{{}^{#1}\kern -0.15em x}
\newcommand{\xlupmax}[1]{{}^{#1}\kern -0.25em x}
\newcommand{\iDelta}{\iota_{\kern -0.075em \Updelta}}
\newcommand{\Upphisub}[1]{\Upphi_{\kern -0.1em #1}}
\newcommand{\aff}{\operatorname{\mathsf{aff}}\nolimits}

\newcommand{\boldnu}{\boldsymbol\nu}

\newcommand{\ul}[1]{\underline{#1}}

\begin{document}

\title[]{Silting and Tilting for Weakly Symmetric Algebras}

\author[]{Jenny August}
\address{Max Planck Institute for Mathematics, Vivatsgasse 7, 53111 Bonn, Germany}
\email{jennyaugust@mpim-bonn.mpg.de}

\author[]{Alex Dugas}
\address{Dept.\ of Mathematics, University of the Pacific,
  Stockton, California 95211}
\email{adugas@pacific.edu}

\begin{abstract} 

If A is a finite-dimensional symmetric algebra, then it is well-known that the only silting complexes in $\Kb(\proj A)$ are the tilting complexes.  In this note we investigate to what extent the same can be said for weakly symmetric algebras.  On one hand, we show that this holds for all tilting-discrete weakly symmetric algebras.  In particular, a tilting-discrete weakly symmetric algebra is also silting-discrete.  On the other hand, we also construct an example of a weakly symmetric algebra with silting complexes that are not tilting.
\end{abstract}

%\subjclass[2010]{Primary~; Secondary~}
%\keywords{Stability conditions, contraction algebra, flopping contraction, $K(\uppi,1)$.}
%\thanks{J.A. was supported by the Carnegie Trust for the Universities of Scotland, and M.W. was supported by EPSRC grants~EP/R009325/1 and EP/R034826/1.}
\maketitle{}

\section{Introduction}

For a finite-dimensional $k$-algebra $A$, the tilting complexes play a central role in the category $\Kb(\proj A)$ of perfect complexes. One of the main tools used in their study is \emph{mutation}, but to get a well-behaved mutation, one is led to consider the weaker notion of silting complexes instead. While the silting theory of $A$ can be quite complicated in general, the notion of \emph{silting-discreteness} was introduced by Aihara \cite{aihara} as a strong finiteness property. This can make it possible to describe all the silting complexes over $A$ and their behavior under mutation. For example, under this condition it is well known that $A$ is \emph{silting connected} \cite{aihara} i.e.\ any two silting complexes of $A$ can be connected by a sequence of mutations. 

The silting-discreteness property also has particularly nice implications on the Bridgeland stability manifold associated to $\Db(\mod A)$ \cite{AMY, PSZ} -- a topological invariant related to the $t$-structures in this category. In particular, Pauksztello--Saorin--Zvonareva show that for a silting-discrete algebra, the bounded  $t$-structures in $\Db(\mod A)$ are in bijection with basic silting complexes in $\Kb(\proj A)$ \cite{PSZ, KY}. Moreover, they use this to show the stability manifold is contractible in this case, something which is often very difficult to determine in more geometric settings.

With this in mind, the results in this article were broadly motivated by the question of which finite-dimensional algebras are silting-discrete. Because of their connection with derived equivalences, it is often easier to control the tilting complexes of an algebra, rather than all the silting complexes. For example, Aihara--Mizuno \cite{AM} use the associated equivalences to show that the preprojective algebras of Dynkin type are tilting-discrete, but it remains an open question whether they are all silting-discrete. In particular, the easiest settings to establish silting-discreteness will be when the notions of silting and tilting (and hence also silting-discrete and tilting-discrete) coincide. This is well-known for symmetric algebras, and so we asked whether the same is true for weakly symmetric algebras. As we show below, if a weakly symmetric algebra is tilting-discrete, then it must also be silting-discrete, and in this case all silting complexes are tilting. In particular, this applies to the weakly symmetric preprojective algebras (those of type $D_{2n}$, $E_7$ and $E_8$), a result which we have since learned was already known to Aihara \cite{aihara2}, although the proof does not explicitly appear in \cite{AM}.  Furthermore, after writing we became aware of work of Adachi and Kase \cite{AK}, which independently proves both of these results as a consequence of a more general theory of $\nu$-stable silting. 
  
However, we additionally return to the question of whether every silting complex over a weakly symmetric algebra is tilting, and we show that the answer is negative in general.  We achieve this by constructing examples of weakly symmetric algebras with silting complexes that are not tilting.  These examples are modifications of the examples of silting-disconnected algebras in \cite{dugas}, and in fact provide further examples of algebras with this property.

\section{Preliminaries}
We let $A$ be a basic finite-dimensional algebra over an algebraically closed field $k$ with $n$ isomorphism classes of simple (right) modules.  We write $e_1, \ldots, e_n$ for a complete set of pairwise orthogonal primitive idempotents for $A$, and write $P_i = e_iA$ for the indecomposable projective right $A$-modules.  We primarily work with right $A$-modules and use $\mod A$ for the category of finitely generated right $A$-modules, $\Db(\mod A)$ for the bounded derived category and $\Kb(\proj A)$ for the homotopy category of perfect complexes over $A$.

\subsection{Twisted modules}

 Let $\sigma$ be a $k$-algebra automorphism of $A$, acting on the left. For any right $A$-module $M$, we define the twisted module $M_{\sigma}$ to be $M$ as a $k$-vector space with the right action of $A$ given by $m \cdot a = m \sigma(a)$ for all $m \in M$ and $a \in A$.  Similarly, for a left $A$-module $N$, we can define the twisted module ${}_\sigma N$ as $N$ but with $A$-action given by $a \cdot m = \sigma(a) m$ for all $a \in A$ and $m \in N$.  Observe that we have natural isomorphisms $M_\sigma \cong M \otimes_A A_{\sigma}$ and ${}_{\sigma} N \cong {}_\sigma A \otimes_A N$ for all right (resp.\ left) $A$-modules $M$ (resp.\ $N$).  Thus $\sigma$ induces an automorphism $\sigma^* \colonequals -\otimes_A A_{\sigma}$ of the category $\mod A$. This action restricts to an automorphism of $\proj A$ and hence also induces automorphisms of $\Kb(\proj A)$ and $\Db(\mod A)$.  

\subsection{Nakayama Automorphism}

Writing $D \colonequals \Hom_k(-,k)$ for the standard $k$-duality between left and right $A$-modules, the \emph{Nakayama functor} $\boldsymbol\nu \colonequals D\Hom_A(-,A)$ is a right exact functor isomorphic to $- \otimes_A DA$. It induces an equivalence 
$\proj A \xrightarrow{\sim} \mathrm{inj} A$
whose quasi-inverse is $\boldnu^{-1} \colonequals \Hom_A(DA,-)$.\\

Recall that $A$ is self-injective if and only if there is an isomorphism of right (or left) modules $A \xrightarrow{\sim} DA$. In this case, all projective modules are injective and vice versa, and there always exists an algebra automorphism $\nu \colon A \to A$ such that there is a isomorphism of $A$-bimodules 
\begin{align*}
\varphi \colon A_\nu \to DA. 
\end{align*}

Note that $\nu$ is unique up to inner automorphism, and we call $\nu$ the \emph{Nakayama automorphism} of $A$ since $\nu^* \colonequals - \otimes_A A_\nu $ coincides with the Nakayama functor $\boldnu$. It is well known that $A$ is symmetric if and only if $\nu$ is inner which is if and only if $\boldnu$ is isomorphic to the identity functor. Note that $\boldnu(P_i) \cong I_i$ for all finite-dimensional algebras, but if $A$ is self-injective, there exists a permutation $\pi$ such that, for all $i$,
\begin{align*}
P_i \cong I_{\pi(i)}, \quad \text{or equivalently}, \quad \boldnu P_i \cong P_{\pi^{-1}(i)}.
\end{align*}
This permutation $\pi$ is known as the \emph{Nakayama permutation} of $A$.
\begin{dfn}
An algebra is \emph{weakly symmetric} if the Nakayama permutation is the identity i.e. $P_i \cong \boldnu P_i$ for all $i$. Or equivalently, if $\nu(e_i) \cong e_i$ for all $i$.
\end{dfn}
Note that the weakly symmetric property is strictly weaker than being symmetric.
\begin{thm}\cite[4.8]{BBK} \label{preproj self injective}
The preprojective algebras of ADE Dynkin type are self-injective. They are weakly symmetric if the Dynkin type is $D_{2n}, E_7,$ or $E_8$ but these are not symmetric unless $\mathrm{char} \ k =2$.
\end{thm}

\subsection{Nakayama and Tilting}
Recall that a complex $T \in \Kb(\proj A)$ is called tilting (resp.\ silting) if 
\begin{enumerate}
\item $\Hom_{\Kb(\proj A)}(T,T[n])=0$ for all $n \neq 0$ (resp.\ for all $n >0$);
\item the smallest full triangulated subcategory of $\Kb(\proj A)$ containing $T$ and closed under forming direct summands is $\Kb(\proj A)$.
\end{enumerate}
We will write $\tilt A$ (resp.\ $\silt A$) for the set of isomorphism classes of basic tilting (resp.\ silting) complexes in $\Kb(\proj A)$. If $A$ is self-injective, the Nakayama functor restricts to an equivalence $\boldnu \colon\proj A \to \proj A$ and hence there is an induced equivalence
\begin{align*}
\boldnu \colon \Kb(\proj A) \to \Kb(\proj A). 
\end{align*}
It follows that if $T$ is a tilting (resp.\ silting) complex, then so is $\boldnu T$. 

\begin{thm}\cite[A.4]{aihara} \label{nakayama stable iff tilting}
If $A$ is self-injective then a basic silting complex $T$ is tilting if and only if $\boldnu T \cong T$.
\end{thm}

Recall that when $A$ is symmetric, $\boldnu \cong \mathrm{id}$, and hence a direct corollary of this result is the well-known fact that all silting complexes over a symmetric algebra are tilting complexes.

%%%%%%%%%%%%%%%%%%%%%%%%%%%%%%%%%%%%%%%%%%%%

\subsection{Silting Mutation} 
To create new silting complexes from a given one, Aihara--Iyama introduced the notion of mutation \cite{AI}.

\begin{dfn}
Suppose that $T = X \oplus Y \in \Kb(\proj A)$ is a basic silting complex. Then consider a triangle 
\begin{align*}
X \xrightarrow{f} Y' \xrightarrow{g} X' \to X[1]
\end{align*}
where $f$ is a left $\add(Y)$-approximation of $X$. Then $\mu_X(T) \colonequals X' \oplus Y$ is a silting complex called the  left mutation of $T$ with respect to $X$. There is a dual notion of right mutation. Such mutations are called \emph{irreducible} if $X$ is indecomposable.
\end{dfn} 
For any finite dimensional algebra $A$, we may view the algebra as a complex centred in degree zero, and this will always be a tilting complex. An algebra $A$ is called \emph{silting connected} (resp.\ \emph{weakly silting connected}) if all basic silting complexes in $\Kb(\proj A)$ can be obtained from $A$ by a sequence of irreducible (resp.\ not necessarily irreducible) mutations, left or right at each stage. Note that not all algebras are weakly silting connected \cite{dugas}.

If $A$ is self-injective, we say that a complex $X \in \Kb(\proj A)$ is \emph{Nakayama stable} if \mbox{$\boldnu(X) \cong X$}. In other words, Theorem \ref{nakayama stable iff tilting} says that a silting complex $T$ is tilting if and only if it is Nakayama stable. We further call $T$ \emph{strongly Nakayama stable} if each indecomposable summand of $T$ is Nakayama stable.
\begin{prop} \label{strong nakayama stable} \cite[2.1]{dugas}
If $A$ is self-injective and $T$ is a strongly Nakayama stable tilting complex, then any (not necessarily irreducible) mutation of $T$ is also strongly Nakayama stable. 
\end{prop}

%%%%%%%%%%%%%%%%%%%%%%%%%%%%%%%%%%%%%%%%%%%%%%%%%%%%

\section{Tilting Theory for Weakly Symmetric Algebras}

In this section, we make some initial observations on the tilting theory of weakly symmetric algebras, before then placing an extra condition on the algebras, known as \emph{tilting-discreteness}, and showing that all silting complexes are tilting in this case. As before, we let $A$ be a basic finite-dimensional algebra over a field $k$ with $n$ isomorphism classes of simple modules.

\subsection{Initial Observations}

Since a basic silting complex $T \in \Kb(\proj A)$ has $n$ indecomposable summands by \cite[2.28]{AI}, we can write $T = \oplus_{i=1}^n T_i$ where each $T_i$ is indecomposable.  If $A$ is self-injective and $T$ is tilting, then Theorem \ref{nakayama stable iff tilting} shows $\boldnu$ must permute these summands. In this case, the associated standard derived equivalence $\Db(\mod A) \to \Db(\mod \End(T))$ maps the $T_i$ to the distinct indecomposable projective modules over $\End(T)$, and commutes with the Nakayama functors of the two algebras \cite[5.2]{rickard}.  Hence it follows that the permutation of the $T_i$ induced by $\boldnu$ will correspond with the Nakayama permutation of $\End(T)$.

\begin{prop} \label{weaklySymTilting} Let $A$ be weakly symmetric.  Then any tilting complex $T \in \Kb(\proj A)$ is strongly Nakayama stable. Consequently, any algebra derived equivalent to $A$ is also weakly symmetric.
\end{prop} 

\begin{proof} 
 The Grothendieck group of the triangulated category $\Kb(\proj A)$ is a free abelian group with basis elements $[P_i]$ for each indecomposable projective $A$-module $P_i$.  Since a tilting complex $T= \oplus_{i=1}^n T_i \in \Kb(\proj A)$ with $B= \End(T)$ induces an equivalence of triangulated categories $\Kb(\proj B) \to \Kb(\proj A)$ taking $B$ to $T$, it induces an isomorphism of Grothendieck groups taking the natural basis over $B$ to $\{[T_i]\}_{i=1}^n$.  Thus the latter is a basis for the Grothendieck group of $\Kb(\proj A)$ (this is in fact true if $T$ is any silting complex by \cite[2.27]{AI}).  However, if $A$ is weakly symmetric, we have $\boldnu P_i \cong P_i$ for all $i$, and thus $\boldnu$ acts as the identity on the Grothendieck group.  Since $\boldnu$ permutes the $T_i$, if $\boldnu T_i \cong T_j$, then $[T_i] = [\boldnu T_i] = [T_j]$ in the Grothendieck group, which means that $T_i = T_j \cong \boldnu T_i$, as required.  The second statement of the proposition, now follows from the fact mentioned above that the action of $\boldnu$ on the $T_i$ induces the Nakayama permutation of $\End(T)$.
\end{proof}
Our next observation is the following direct corollary of Proposition \ref{strong nakayama stable}.
\begin{prop} \label{strong reachable complexes}
If $A$ is a weakly symmetric algebra, then all silting complexes reachable from $A$ via iterated mutation are strongly Nakayama stable tilting complexes. Moreover, their endomorphism algebras will all be weakly symmetric algebras.
\end{prop}
\begin{proof}
Since $A$ is weakly symmetric, by definition we have $\boldnu P_i \cong P_i$ for all indecomposable projective modules and thus $A$ is a strongly Nakayama stable tilting complex. By Proposition \ref{strong nakayama stable}, any mutation of $A$ is again a strongly Nakayama stable tilting complex, and hence iterating this result shows any silting complex reachable from $A$ is a strongly Nakayama stable tilting complex. The second statement again uses the fact mentioned above that the action of $\boldnu$ on a tilting complex $T$ induces the Nakayama permutation of $\End(T)$.
\end{proof}

\begin{cor} \label{silting connected} 
If $A$ is weakly symmetric and weakly silting connected, then every silting complex for $A$ is a strongly Nakayama stable tilting complex.
\end{cor}
\begin{proof}
By Proposition \ref{strong reachable complexes}, since $A$ is weakly symmetric, all silting complexes reachable from $A$ by mutation are strongly Nakayama stable tilting complexes. Since $A$ is weakly silting connected, these are all the silting complexes of $A$ and hence the result follows. 
\end{proof}

\subsection{Tilting-discreteness}

Silting- and tilting-discreteness are notions which were developed by Aihara--Mizuno \cite{AM} using the partial order on silting complexes introduced by Aihara-Iyama \cite{AI}.

\begin{dfn}
If $T, S \in \Kb(\proj A)$ are two silting complexes, then we say $T \geq S$ if $\Hom_{\Kb(\proj A)}(T,S[n])=0$ for all $n > 0$.
\end{dfn}
This partial order is one of the key tools used when studying silting theory, as it is known to control mutation. It may also be used to define a certain subset of silting complexes, often studied because of their connections with $\tau$-tilting theory and cluster-tilting theory.

\begin{dfn}
A basic silting complex $T \in \Kb(\proj A)$ is called \emph{two-term} if $A \geq T \geq A[1]$ or equivalently, $T$ only has nonzero terms in degrees $0$ and $-1$.
\end{dfn}

\begin{prop} \label{two-term}
For a weakly symmetric algebra $A$, all two-term silting complexes are tilting.
\end{prop}
\begin{proof}
Suppose that $A = \bigoplus_{i=1}^n P_i$ and that $T$ is a two-term silting complex for $A$. Then, since the $[P_i]$ give a basis of the Grothendieck group of $\Kb(\proj A)$, we may write 
\begin{align*}
[T] = \bigoplus_{i=1}^n a_i[P_i]
\end{align*}
and, using the language of \cite{dij}, we say that the \emph{$g$-vector} of $T$ is $(a_1, \dots, a_n) \in \mathbb{Z}^n$. Now $\boldnu T$ is another two-term silting complex for $A$, and since $A$ is weakly symmetric ($\boldnu P_i \cong P_i$ for all $i$),  $\boldnu T$ must have the same $g$-vector. However, by \cite[6.5]{dij}, $g$-vectors completely determine two-term silting complexes and thus $T \cong \boldnu T$ and $T$ is tilting by Theorem \ref{nakayama stable iff tilting}.
\end{proof}
If an algebra $A$ has finitely many basic two-term silting complexes, the algebra is called \emph{$\tau$-tilting finite}. Aihara \cite{aihara} generalised this notion, with Aihara--Mizuno then developing it further.
\begin{dfn}\cite[2.4, 2.11]{AM}
A self-injective finite-dimensional algebra $A$ is called tilting-discrete (resp.\ silting-discrete) if the set
\begin{align*}
\{ T \in \tilt A \mid P \geq T \geq P[1]\} \quad \text{(resp. $\{ T \in \silt A \mid P \geq T \geq P[1]\}$)}
\end{align*}
is finite for any tilting (resp.\ silting) complex $P$ obtained from $A$ by iterated irreducible left mutation.
\end{dfn}
It is clear that silting-discrete implies tilting-discrete and if the algebra $A$ is symmetric, the two notions are equivalent. It is also known that silting-discrete implies silting connected \cite[3.9]{AM} and tilting-discrete implies tilting-connected \cite[5.14]{simpleminded}. However, if we only know an algebra is tilting-discrete, it is generally unknown whether the algebra is also silting-discrete/silting connected.
\begin{prop}\label{tdwsa} \emph{(Cf.\ \cite[Cor.\ 2.25]{AK})} 
If $A$ is a tilting-discrete weakly symmetric algebra, then $A$ is in fact silting-discrete and all silting complexes for A are tilting.
\end{prop}
\begin{proof}
Suppose that $P$ is a silting complex obtained from $A$ by iterated irreducible left mutation. Then, since $A$ is weakly symmetric, $P$ is a strongly Nakayama stable tilting complex, and $B\colonequals \End_A(P)$ is a weakly symmetric algebra using Proposition \ref{strong reachable complexes}. Thus, there is a standard derived equivalence 
\begin{align*}
F \colon \Db(\mod A) &\to \Db(\mod B)\\
P &\mapsto B
\end{align*} 
and this preserves silting (resp.\ tilting) complexes and the silting order (see e.g.\ \cite[2.8]{August1}). In particular, $F$ induces a bijection
\begin{align}
\{ T \in \silt A \mid P \geq T \geq P[1]\} \leftrightarrow \{ S \in \silt B \mid B \geq S \geq B[1]\} \label{silt sets}
\end{align}
which further restricts to a bijection 
\begin{align}
\{ T \in \tilt A \mid P \geq T \geq P[1]\} \leftrightarrow \{ S \in \tilt B \mid B \geq S \geq B[1]\}. \label{tilt sets}
\end{align}
By the tilting-discreteness of $A$, the left hand side of \eqref{tilt sets} is finite and hence so is the right hand side. However, as $B$ is weakly symmetric, Proposition \ref{two-term} shows that 
\begin{align*}
\{ S \in \tilt B \mid B \geq S \geq B[1]\} = \{ S \in \silt B \mid B \geq S \geq B[1]\}
\end{align*}
and thus, both sides in \eqref{silt sets} are also finite, proving that $A$ is silting-discrete. Then by \mbox{\cite[3.9]{aihara}}, this implies $A$ is silting connected and thus all silting complexes can all be obtained from $A$ by iterated mutation. Using Proposition \ref{silting connected} this shows that all silting complexes are strongly Nakayama stable tilting complexes.
\end{proof}

\begin{cor} \emph{(Cf.\ \cite[Ex.\ 22]{aihara2}, \cite[Ex.\ 2.26]{AK})}
The preprojective algebras of Dynkin type $D_{2n}, E_7$ and $E_8$ are silting-discrete algebras, where every silting complex is a tilting complex.
\end{cor}

\begin{proof}
By Theorem \ref{preproj self injective}, these algebras are weakly symmetric and \cite[5.1]{AM} shows that they are tilting-discrete. The result then follows directly from Proposition \ref{tdwsa}. 
\end{proof}
One application of silting-discreteness is in the study of Bridgeland stability. Given a triangulated category, in this case the bounded derived category of our finite dimensional algebra, Bridgeland stability constructs a complex manifold associated to this category. If $A$ is a finite dimensional silting-discrete algebra, then \cite{PSZ} show that this manifold will be contractible, and combining this with Proposition \ref{tdwsa} immediately gives the following. 
\begin{cor}
If $A$ is a finite dimensional weakly symmetric tilting-discrete algebra, then the Bridgeland stability manifold of $\Db(\mod A)$ is contractible.
\end{cor}
\begin{proof}
This follows directly from Proposition \ref{tdwsa} and \cite{PSZ}.
\end{proof}

\section{Examples}

We now give examples of weakly symmetric algebras with silting complexes that are not tilting.  The examples are based on those in \cite{dugas}, so we begin by reviewing the necessary details from that work.

We fix an even integer $n \geq 4$ and let $A$ be the path algebra of the quiver
$$\xymatrix{1 \ar[r]<0.5ex>^x \ar[r]<-0.5ex>_y & 2 \ar[r]<0.5ex>^x \ar[r]<-0.5ex>_y & \cdots \ar[r]<0.5ex>^x \ar[r]<-0.5ex>_y & n}$$
modulo the relations $x^2=y^2=0$.  We write $e_i$ for the primitive idempotent of $A$ corresponding to vertex $i$ (for $1 \leq i \leq n$).
As $A$ has finite global dimension, we can identify $\Kb(\proj A)$ with $\Db(\mod A)$, and we write $\mathbb{S} := -\otimes^{\bL}_A DA$ for the Serre functor on this category.

We let $\sigma \in \Aut_k(A)$ be the order two automorphism induced by the automorphism of $Q$ that fixes each vertex and swaps each pair of $x$ and $y$ arrows.  We write $\sigma^*$ for the induced automorphisms on the categories $\mod A$, $\Kb(\proj A)$, or $\Db(\mod A)$ depending on context.  %We say that an object $M$ in one of these categories is $\sigma^*$-invariant if $\sigma^*X \cong X$ for each indecomposable summand $X$ of $M$.  Clearly we have $\sigma(e_i) = e_i$ for all $i$, and thus each indecomposable projective $P_i$ is $\sigma^*$-invariant.  
We set $E = e_1A/e_1yA$, which is a uniserial module of length $n$, and note that $\sigma^* E \cong e_1A/e_1xA \ncong E$.
  %When we consider $E \in D^b(\mod A) \approx K^b(\proj A)$ we may identify $E$ with its minimal projective resolution 
%$$P_E = 0 \to P_n \stackrel{y\cdot}{\to} P_{n-1} \stackrel{y\cdot}{\to} \cdots \to P_2 \stackrel{y\cdot}{\to} P_1 \to 0,$$
%which has $P_1$ in degree $0$.

\begin{prop} \cite[4.1]{dugas} \label{prop:ESpherical} $E$ and $\sigma^* E$ are Hom-orthogonal $(n-1)$-spherical objects in $\Db(\mod A)$.  
\end{prop}

Now $E$ defines a spherical twist functor, which we can apply to $A$ to obtain a tilting complex $T$ that fits into an exact triangle 
\begin{equation} \label{eq:defT} E[1-n]^{n} \to A \to T \to E[2-n]^n\end{equation} in $\Db(\mod A)$.
By applying $\sigma^*$, and using the fact that $\sigma^*A \cong A$ we obtain another triangle
\begin{equation} \label{eq:defT2} \sigma^*E[1-n]^{n} \to A \to \sigma^*T \to \sigma^*E[2-n]^n.\end{equation}

To get a weakly symmetric algebra, we can form the \emph{twisted trivial extension} of $A$ using the automorphism $\sigma$.  Thus we define $\Lambda \colonequals T_\sigma A = A \ltimes {}_{\sigma} DA$, where the latter denotes the usual bimodule extension of $A$ by the bimodule ${}_{\sigma} DA$.  The idempotents $e_i$ of $A$ induce a complete set of primitive orthogonal idempotents $(e_i,0)$ of $\Lambda$, which we will continue to write as $e_i$.  In general, by \cite[Prop. 2.2]{grant} the Nakayama automorphism $\nu$ of $T_{\sigma}A$ is given by 
\begin{equation} \label{eq:NakayamaAut} \nu (a, f) = (\sigma(a), f \sigma^{-1}). \end{equation}
In particular, since $\sigma$ fixes the idempotents $e_i$ of $A$, we see also that $\nu(e_i) = e_i$ for all $i$, and thus $\Lambda$ is weakly symmetric.  The quiver and relations of a twisted trivial extension can be computed as described in \cite[\S 3]{guo}, for example.  In our case, we see that $\Lambda$ has quiver
$$
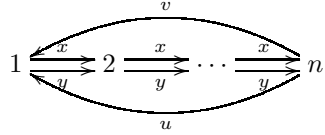
\xymatrix{1 \ar[r]<0.5ex>^x \ar[r]<-0.5ex>_y & 2 \ar[r]<0.5ex>^x \ar[r]<-0.5ex>_y & \cdots \ar[r]<0.5ex>^x \ar[r]<-0.5ex>_y & n  \ar@/^1.5pc/[lll]^u \ar@/_1.5pc/[lll]_v}$$
with relations $x^2 = y^2 = 0$ and $xv = ux = yu = vy = 0$, together with additional relations expressing equality of the two (remaining) nonzero paths of length $n$ at each vertex:  for $0 \leq r < q := n/2$,
$$(xy)^r v (xy)^{q-r-1}x = (yx)^r u (yx)^{q-r-1}y\ \mbox{and}  \ (xy)^r x u (yx)^{q-r-1} = (yx)^r y v (xy)^{q-r-1}.$$
Furthermore, we can see from (\ref{eq:NakayamaAut}) that $\nu$ swaps each pair $x$ and $y$ (of parallel arrows), while also swapping $u$ and $v$.  For $n=4$, the tilting complex $T \in \Kb(\proj A)$ from \eqref{eq:defT} is described in \cite{dugas}.  The corresponding complex $T \otimes_A \Lambda \in \Kb(\proj \Lambda)$ will look the same, but with each $A$ replaced by $\Lambda$.  Its indecomposable summands are as follows, where we
indicate the degree-0 term by underlining it:

$$\xymatrix{ 0 \ar[r] & 0 \ar[r] & \ul{e_3\Lambda} \ar[r]^{y}  & e_2\Lambda \ar[r]^{y} &  e_1\Lambda \ar[r] & 0 \\
  0 \ar[r] & e_4\Lambda \ar[r]^(.4){\scriptsize \begin{pmatrix}x \\ y \end{pmatrix}} & \ul{(e_3\Lambda)^2} \ar[r]^(.55){\scriptsize \begin{pmatrix}0 & y\end{pmatrix}} & e_2\Lambda \ar[r]^{y} & e_1 \Lambda \ar[r] & 0 \\
 0 \ar[r] & e_4\Lambda \ar[r]^(0.35){\scriptsize \begin{pmatrix}yx \\ y \end{pmatrix}} & \ul{e_2\Lambda \oplus e_3 \Lambda} \ar[r]^(0.65){\scriptsize \begin{pmatrix}0 & y\end{pmatrix}} & e_2\Lambda \ar[r]^{y} & e_1 \Lambda \ar[r] & 0 \\
 0 \ar[r] &  e_4\Lambda \ar[r]^(0.35){\scriptsize \begin{pmatrix}xyx \\ y \end{pmatrix}} & \ul{e_1\Lambda \oplus e_3 \Lambda} \ar[r]^(0.65){\scriptsize \begin{pmatrix}0 & y\end{pmatrix}} & e_2\Lambda \ar[r]^{y} & e_1 \Lambda \ar[r] & 0}$$

As this complex is clearly not invariant under the Nakayama functor $\boldnu$, it is not a tilting complex.  However, it is silting.

\begin{prop}  Let $A, \sigma$ and $T$ be as above, and let $\Lambda = T_{\sigma}A$.  Then $T \otimes_A \Lambda$ is a silting complex in $\Kb(\proj \Lambda)$ that is not tilting.
\end{prop}

\begin{proof}  
The proof is similar to Rickard's that $T \otimes_A TA$ is a tilting complex over the trivial extension algebra $TA$ for any tilting complex $T$ over $A$ \cite{DCSE}. We begin by noting that $T \otimes_A \Lambda$ generates $\Kb(\proj \Lambda)$. This can be seen using that $T$ generates $\Kb(\proj A)$ and $-\otimes_A \Lambda : \Kb(\proj A) \to \Kb(\proj \Lambda)$ is an exact functor of triangulated categories taking $A$ to $\Lambda$.  It remains to show that 
$$\Hom_{\Kb(\Lambda)}(T\otimes_A \Lambda, T \otimes_A \Lambda [i]) = 0 \ \mbox{for\ all}\ i>0.$$  To this end, observe that for all $i \neq 0$
\begin{eqnarray*} \Hom_{\Kb(\Lambda)}(T \otimes_A \Lambda, T \otimes_A \Lambda [i]) & \cong & \Hom_{\Kb(A)}(T, T\otimes_A \Lambda[i]) \\
& = & \Hom_{\Kb(A)}(T,T[i] \oplus T \otimes_A {}_{\sigma}DA[i]) \\ & \cong & \Hom_{\Kb(A)}(T,T[i]) \oplus \Hom_{\Kb(A)}(T,(\sigma^*)^{-1} \mathbb{S}T[i]) \\
& \cong & 0 \oplus \Hom_{\Kb(A)}(\sigma^* T, \mathbb{S}T[i]) \\
& \cong & D\Hom_{\Kb(A)}(T, \sigma^* T[-i]),
\end{eqnarray*}
where the penultimate isomorphism is from the fact that $T_A$ is a tilting complex and the last is by Serre duality.  Thus it suffices to show that $\Hom_{\Kb(A)}(T, \sigma^*T[j]) = 0$ for all $j<0$.

   For the remainder of the proof, we are working in the category $\Kb(\proj A)$, and so we will omit the corresponding subscripts in our Hom-spaces.  Applying $\Hom(-,\sigma^*E[j])$ to (\ref{eq:defT}) and using the fact that $E$ and $\sigma^*E$ are Hom-orthogonal, we get isomorphisms
\begin{equation} \label{eq:iso} \Hom(T,\sigma^*E[j]) \cong \Hom(A,\sigma^*E[j]) \end{equation}
 for all $j$, and the latter vanishes for all $j \neq 0$ since the homology of $\sigma^*E$ is concentrated in degree $0$.
Now we apply $\Hom(T,-)$ to (\ref{eq:defT2}), which yields isomorphisms $$\Hom(T,A[j]) \cong \Hom(T,\sigma^*T[j])$$ for all $j \neq n-2,n-1$.  In particular, for all $j<0$, we have $\Hom(T,\sigma^*T[j]) \cong \Hom(T,A[j])$.   
We now show that $\Hom(T,A[j]) = 0$ for $j<0$.  Apply $\Hom(-,A[j])$ to (\ref{eq:defT}) to get an exact sequence
\begin{equation}\label{eq:seq} \Hom(E[2-n]^n,A[j]) \to \Hom(T,A[j]) \to \Hom(A,A[j]). \end{equation}
 By Serre duality, the first term is isomorphic to $$D\Hom(A[j],\mathbb{S}(E)[2-n]^n) \cong D\Hom(A[j],E[1]^n) \cong D\Hom(A,E[1-j]^n) = 0$$ for $j \neq 1$.  As the last term of (\ref{eq:seq}) vanishes for $j \neq 0$, we see that $\Hom(T,A[j]) = 0$ for all $j<0$, as required (in fact, for all $j \neq 0,1$).

Consequently, $\Hom(T,\sigma^*T[j]) = 0$ for all $j<0$ as required.  Thus $T \otimes_A \Lambda$ is a silting complex.  While, we can see that $T \otimes_A \Lambda$ is not a tilting complex since it is not invariant under the Nakayama functor of $\Lambda$, which switches $x$ and $y$, we also provide a direct proof by showing that it has nonzero self-extensions in degree $2-n$. 

Applying $\Hom(T,-)$ to  (\ref{eq:defT2}), and using $\Hom(T,A[j]) = 0$ for $j \neq 0,1$, and then (\ref{eq:iso}), gives $$\Hom(T,\sigma^*T[n-2]) \cong \Hom(T, \sigma^*E^n) \cong \Hom(A,\sigma^*E^n) \cong \sigma^*E^n.$$

Thus $$\Hom_{\Kb(\Lambda)}(T \otimes_A \Lambda, T \otimes_A \Lambda[2-n]) \cong D\Hom(T,\sigma^*T[n-2]) \cong D(\sigma^*E)^n \neq 0.  $$ 

\end{proof}

As a consequence of Proposition \ref{tdwsa}, the algebra $\Lambda$ is not tilting-discrete.  In fact, combining with Corollary \ref{silting connected}, we see that it is not even weakly silting connected.  We conclude by pointing out another interesting property of the silting complex $T$, which follows from Propositions \ref{weaklySymTilting} and \ref{strong nakayama stable}, and to our knowledge has not been observed in other examples.

\begin{cor}  For $\Lambda$ and $T$ as defined above, $T$ is a silting complex which is not connected to any tilting complex by iterated silting mutations.
\end{cor}

\end{document}